\documentclass[a4paper,11pt]{amsart}
\usepackage{amsopn, amssymb, amscd, amsmath, amsthm}
\usepackage{graphicx, graphics, epsfig}

\newcommand{\nc}{\newcommand}

\nc{\RS}{\operatorname{RicSign}} \nc{\REV}{\operatorname{RicEV}}

\nc{\Gl}{\mathsf{GL}} \nc{\Or}{\mathsf{O}}  \nc{\SO}{\mathsf{SO}}  \nc{\SU}{\mathsf{SU}}  \nc{\Sl}{\mathsf{SL}}
\nc{\G}{\mathsf{G}} \nc{\K}{\mathsf{K}}  \nc{\T}{\mathsf{T}} \nc{\Lsf}{\mathsf{L}}
\nc{\Qb}{\mathsf{Q}_\Beta} \nc{\Hb}{\mathsf{H}_\Beta} \nc{\Ub}{\mathsf{U}_\Beta}
\nc{\Gb}{\mathsf{G}_\Beta} \nc{\Kb}{\mathsf{K}_\Beta}
\nc{\PPP}{\mathsf{P}} \nc{\U}{\mathsf{U}} \nc{\N}{\mathsf{N}} \nc{\Ss}{\mathsf{S}} \nc{\Aa}{\mathsf{A}}

\nc{\laH}{\la\!\la} \nc{\raH}{\ra\!\ra}

\nc{\ipH}{{\laH \cdot, \cdot \raH}}

\nc{\Vg}{{V(\ggo)}}

\nc{\alert}{\color{blue}}

\nc{\fg}{\mathfrak{f}}  \nc{\vg}{\mathfrak{v}} \nc{\wg}{\mathfrak{w}} \nc{\zg}{\mathfrak{z}} \nc{\ngo}{\mathfrak{n}} \nc{\kg}{\mathfrak{k}} \nc{\mg}{\mathfrak{m}} \nc{\bg}{\mathfrak{b}} \nc{\ggo}{\mathfrak{g}} \nc{\ggob}{\overline{\mathfrak{g}}} \nc{\sog}{\mathfrak{so}} \nc{\sug}{\mathfrak{su}} \nc{\spg}{\mathfrak{sp}} \nc{\slg}{\mathfrak{sl}} \nc{\glg}{\mathfrak{gl}} \nc{\cg}{\mathfrak{c}} \nc{\rg}{\mathfrak{r}}  \nc{\hg}{\mathfrak{h}} \nc{\tgo}{\mathfrak{t}} \nc{\ug}{\mathfrak{u}} \nc{\dg}{\mathfrak{d}} \nc{\ag}{\mathfrak{a}} \nc{\pg}{\mathfrak{p}} \nc{\sg}{\mathfrak{s}} \nc{\affg}{\mathfrak{aff}} \nc{\qg}{\mathfrak{q}}
\nc{\Xg}{\mathfrak{X}} \nc{\lgo}{\mathfrak{l}}

\nc{\pca}{\mathcal{P}} \nc{\nca}{\mathcal{N}} \nc{\lca}{\mathcal{L}} \nc{\oca}{\mathcal{O}} \nc{\mca}{\mathcal{M}} \nc{\tca}{\mathcal{T}} \nc{\aca}{\mathcal{A}} \nc{\cca}{\mathcal{C}} \nc{\gca}{\mathcal{G}} \nc{\sca}{\mathcal{S}} \nc{\hca}{\mathcal{H}} \nc{\bca}{\mathcal{B}} \nc{\dca}{\mathcal{D}}

\nc{\vp}{\varphi} \nc{\ddt}{\tfrac{{\rm d}}{{\rm d}t}} \nc{\dds}{\tfrac{{\rm d}}{{\rm d}s}} \nc{\ddtbig}{\frac{{\rm d}}{{\rm d}t}} \nc{\dd}{{\rm d}}
\nc{\dpar}{\tfrac{\partial}{\partial t}} \nc{\im}{\mathtt{i}}

\nc{\RR}{{\mathbb R}} \nc{\HH}{{\mathbb H}} \nc{\CC}{{\mathbb C}} \nc{\ZZ}{{\mathbb Z}}
\nc{\FF}{{\mathbb F}} \nc{\NN}{{\mathbb N}} \nc{\QQ}{{\mathbb Q}} \nc{\PP}{{\mathbb P}}

\nc{\vs}{\vspace{.2cm}} \nc{\vsp}{\vspace{1cm}} \nc{\ip}{{\langle\cdot,\cdot\rangle}}
\nc{\ipp}{(\cdot,\cdot)} \nc{\la}{\langle} \nc{\ra}{\rangle} \nc{\unm}{\tfrac{1}{2}}
\nc{\unc}{\tfrac{1}{4}} \nc{\und}{\tfrac{1}{16}} \nc{\no}{\vs\noindent}
\nc{\lam}{\Lambda^2(\RR^n)^*\otimes\RR^n} \nc{\tangz}{{\rm T}^{\rm Zar}}
\nc{\lamg}{\Lambda^2\ggo^*\otimes\ggo}
\nc{\nor}{{\sf n}}  \nc{\mum}{/\!\!/} \nc{\kir}{/\!\!/\!\!/}
\nc{\Ri}{\tfrac{4\Ric_{\mu}}{||\mu||^2}} \nc{\ds}{\displaystyle}
\nc{\lb}{[\cdot,\cdot]} \nc{\isn}{\tfrac{1}{||v||^2}}
\nc{\gkp}{(\ggo=\kg\oplus\pg,\ip)} \nc{\ukh}{(\ug=\kg\oplus\hg,\ip)}
\nc{\tgkp}{(\tilde{\ggo}=\kg\oplus\pg,\ip)}
\nc{\wt}{\widetilde}
\nc{\raw}{\rightarrow} \nc{\lraw}{\longrightarrow} \nc{\hqn}{\mathcal{H}_{q,n}}

\nc{\Spec}{\operatorname{Spec}} \nc{\Nat}{\operatorname{nat}}

\nc{\ad}{\operatorname{ad}}  \nc{\Aut}{\operatorname{Aut}}   \nc{\Inn}{\operatorname{Inn}}   \nc{\Lie}{\operatorname{Lie}} \nc{\Ad}{\operatorname{Ad}} \nc{\Der}{\operatorname{Der}} \nc{\rad}{\operatorname{r}} \nc{\kf}{\operatorname{B}}

\nc{\End}{\operatorname{End}} \nc{\rank}{\operatorname{rank}} \nc{\Ker}{\operatorname{Ker}} \nc{\tr}{\operatorname{tr}} \nc{\sym}{\operatorname{sym}} \nc{\diag}{\operatorname{diag}} \nc{\proy}{\operatorname{pr}} \nc{\Adj}{\operatorname{Adj}} \nc{\vspan}{\operatorname{span}}

\nc{\Hess}{\operatorname{Hess}}  \nc{\dif}{\operatorname{d}} \nc{\sen}{\operatorname{sen}} \nc{\grad}{\operatorname{grad}} \nc{\Order}{\operatorname{O}} \nc{\divg}{\operatorname{div}}

\nc{\Iso}{\operatorname{Iso}} \nc{\Diff}{\operatorname{Diff}} \nc{\ricci}{\operatorname{ric}}  \nc{\Rc}{\operatorname{Rc}} \nc{\Ricci}{\operatorname{Ric}} \nc{\Riem}{\operatorname{Rm}} \nc{\scalar}{\operatorname{sc}} \nc{\scalarm}{\hat{\operatorname{R}}} \nc{\riccim}{\widehat{\operatorname{Ric}}} \nc{\tang}{\operatorname{T}} \nc{\vol}{\operatorname{vol}}

\nc{\mm}{\operatorname{M}} \nc{\CH}{\operatorname{CH}} \nc{\Irr}{\operatorname{Irr}} \nc{\mcc}{\operatorname{mcc}} \nc{\m}{\operatorname{m}} \nc{\pr}{\operatorname{pr}}

\newcommand{\R}{{\mathbb{R}}}

\nc{\Id}{\operatorname{Id}}  \nc{\mmm}{\operatorname{m}}

\numberwithin{equation}{section}

\theoremstyle{plain}
\newtheorem{theorem}{Theorem}[section]
\newtheorem{proposition}[theorem]{Proposition}
\newtheorem{corollary}[theorem]{Corollary}
\newtheorem{lemma}[theorem]{Lemma}

\theoremstyle{definition}

\theoremstyle{remark}
\newtheorem{remark}[theorem]{Remark}

\newtheorem{example}[theorem]{Example}

\title[The prescribed Ricci curvature problem for naturally reductive metrics]{The prescribed Ricci curvature problem for naturally reductive metrics on compact Lie groups}

\author{Romina M.~Arroyo}
\address{\parbox{\linewidth}{FaMAF $\&$ CIEM, Universidad Nacional de C\'ordoba, C\'ordoba, Argentina \\ School of Mathematics and Physics, The University of Queensland, St~Lucia,~QLD, Australia}\vspace{2pt}}
\email{arroyo@famaf.unc.edu.ar,~r.arroyo@uq.edu.au}

\author{Artem Pulemotov}
\address{School of Mathematics and Physics, The University of Queensland, St Lucia, QLD, Australia}
\email{a.pulemotov@uq.edu.au}

\author{Wolfgang Ziller}
\address{Department of Mathematics,
	University of Pennsylvania,
	Philadelphia, PA, USA}
\email{wziller@math.upenn.edu}

\thanks{This research was supported by the Australian Government through the Australian Research Council's Discovery Projects funding scheme (project DP180102185).  The third author was also supported by the NSF  grant 1506148 and a Reybould Fellowship.}

\begin{document}

\maketitle

\begin{center}
{\it Dedicated to the memory of Yuri Berezansky}
\end{center}

	\begin{abstract}
		We study the problem of prescribing the Ricci curvature in the class of naturally reductive metrics on a compact Lie group. We derive necessary as well as sufficient conditions for the solvability of the equations and provide a series of examples.
	\end{abstract}


\section{Introduction}


The prescribed Ricci curvature problem consists in finding a Riemannian metric $g$ on a manifold $M$ such that
\begin{align}\label{PRC_noc}
\Ricci_g=T
\end{align}
for a given (0,2)-tensor field~$T$. The first major result on this problem is due to  DeTurck~\cite{DeTurck}, who proved that one can solve the equation locally if $T$ is non-degenerate. More can be said if $T$ is itself a Riemannian metric. For example, in~\cite{DeTurck-Koiso} it was shown that the equation has no global solutions if $M$ is compact and $T$ has sectional curvature less than~$\frac1{n-1}$. As a consequence, for every  metric $T$, one can find a constant $c_0(T)$ such that no $g$ satisfies
\begin{align}\label{Ric_c}
\Ricci_g=cT
\end{align}
when $c>c_0(T)$. On a compact manifold, it is thus natural, instead of~\eqref{PRC_noc}, to solve~\eqref{Ric_c} treating $c$ as one of the unknowns. This paradigm was advocated by DeTurck in~\cite{DeTurck85} and Hamilton in~\cite{Hamilton84}. Notice that $c$ is necessarily nonnegative if $T$ is a homogeneous metric since a compact homogeneous space cannot have negative-definite Ricci curvature~\cite[Theorem~1.84]{Bss}. Moreover, if $c=0$ and $g$ is homogeneous,  then $g$ must be flat~\cite[Theorem~7.61]{Bss}.

There is little known about the problem unless one makes further symmetry assumptions. It is natural to suppose that the metric $g$ and the tensor $T$ are invariant under a Lie group $G$ acting on $M$. The first such work, due to Hamilton, appeared in~\cite{Hamilton84}, where he showed  that for a left-invariant metric $T$ on $\SU(2)$ there exists a left-invariant metric $g$, unique up to scaling, such that~\eqref{Ric_c} is satisfied with~$c>0$. Buttsworth extended this result in~\cite{Buttsworth19} to all signatures of $T$ and all unimodular Lie groups of dimension~3. In this case, the equation may fail to have a solution. The paper~\cite{AP16b} investigates~\eqref{Ric_c} when $M$ is a manifold with two boundary components and the orbits of $G$ are hypersurfaces. We refer to~\cite[Section~2]{BP19} for an overview of other results.

If $M$ is a homogeneous space $G/H$, the problem has been studied extensively. Most results rely on the the following fact proven in~\cite{AP16}: solutions to~\eqref{Ric_c} are precisely the critical points of the scalar curvature functional $S$ on the set of $G$-invariant metrics, subject to the constraint $\tr_gT=1$. A range of methods were developed in~\cite{AP16,MGAP18,AP19} to show that, in some cases,~\eqref{Ric_c} has a solution by showing that $S$ attains its global maximum. This approach has been very successful for several classes of homogeneous spaces; see, e.g.,~\cite{AP16,MGAP18,AP19}. In general, however, complete solvability results are only available when the isotropy representation of $G/H$ has two irreducible summands; see~\cite{AP16}. We refer to~\cite{BP19} for a survey.

Naturally reductive metrics are a generalisation of normal homogeneous metrics; however, they comprise a much larger set. They have been considered by a number of authors for various geometric applications; see, e.g.,~\cite{DZ79,GS10,EL19}. Given a subgroup $K<G$, denote by $\mathcal M_K$ the set of all left-invariant naturally reductive metrics on $G$ that are invariant under right translations by~$K$. It was shown in~\cite{DZ79} that, on a compact simple Lie group $G$, every left-invariant naturally reductive metric on $G$ lies in $\mathcal M_K$ for some~$K$. 
In this paper, given $T\in \mathcal M_K$, we want to find  solutions to~\eqref{Ric_c} that lie in~$\mathcal M_K$. However, if the isotropy representation of $G/K$ is reducible, then~\eqref{Ric_c} is equivalent to an overdetermined system of algebraic equations. In this case, one cannot expect substantial existence results. We thus assume that $G/K$ is isotropy irreducible and that $G$ is simple. We will show that, when these assumptions hold, solutions to~\eqref{Ric_c} are again critical points of the scalar curvature functional $S\colon \mathcal M_K\to \R$ subject to $\tr_gT=1$.

We prove that, under some simple conditions on $T\in \mathcal M_K$, the functional $S$, subject to the constraint $\tr_gT=1$, attains its global maximum (see Theorem~\ref{sufcon}) and hence~\eqref{Ric_c} has a solution. We also prove that there is a class of metrics $T\in \mathcal M_K$ for which no critical point exists (see Theorem~\ref{neccon}). These results do not solve the problem completely. Figure~1 indicates, for a specific choice of $G$ and~$K$, the regions where the necessary and the sufficient conditions are satisfied. We also indicate in the same example the precise region where solutions exist. This was found with the help of Renato Bettiol by implementing the cylindrical algebraic decomposition algorithm; see~\cite{CD98}.

We exhibit the global behaviour of the scalar curvature functional by drawing its graph in some special cases. In one of our examples, $S$ has a global maximum without the sufficient condition of Theorem~\ref{sufcon} being satisfied; in another, no critical point exists despite the necessary condition of Theorem~\ref{neccon} being satisfied. As we explain at the end of paper, the former example demonstrates an interesting behaviour that has not been observed before on homogeneous spaces. If $G$ and $K$ are both simple, we are able to provide a complete answer to the question of solvability of~\eqref{Ric_c}; see Proposition~\ref{prop_simple_gr}.

\vs \noindent {\it Acknowledgements.} We thank Renato Bettiol for implementing the cylindrical algebraic decomposition algorithm to help us produce Figure~\ref{T_chart_T1T2}. We are also grateful to Mark Gould and Ramiro Lafuente for helpful discussions. The third author would like to thank the University of Queensland for its hospitality and the Ethel Harriet Raybould Trust for its financial support.

\section{Preliminaries}\label{secnat}

In this section, we recall the characterisation of naturally reductive metrics on a simple Lie group. We then present formulas for their Ricci curvature and scalar curvature.

\subsection{Naturally reductive metrics}\label{subsec_mf_nat_red}

Let $K<G$ be two non-trivial compact Lie groups with Lie algebras $\kg$ and~$\ggo$. It will be convenient for us to assume that $G$ is simple and $K\ne G$; however, see Remark~\ref{semisimple}. We can choose the negative of the Killing form as the background inner product on $\ggo$, which we denote by~$Q$. We have the $Q$-orthogonal splitting
\begin{equation}\label{deck} 
\ggo=\ag \oplus \zg(\kg) \oplus \kg_1 \oplus \cdots \oplus \kg_r,
\end{equation}
where $\ag$ is the complement $\kg^\perp$, $\zg (\kg)$ is the center of $\kg$, and $\kg_i$ are simple ideals of $\kg$.
Consider a left-invariant metric $g$ on~$G$. We identify this metric with the inner product it induces on~$\ggo$. Assume that 
\begin{equation}\label{metric_gen}
g = \alpha Q|_{\ag} +h + \alpha_1 Q |_{\kg_1} + \cdots + \alpha_r Q |_{\kg_r}.
\end{equation}
Here, $h$ is an inner product on $\zg(\kg)$, and $\alpha,\alpha_1,\ldots,\alpha_r$ are positive constants. We can always diagonalize $h$ in a $Q$-orthonormal basis of~$\zg(\kg)$. This yields a further splitting
\begin{equation}\label{centre_dec}
\zg(\kg) = \kg_{r+1} \oplus \cdots \oplus \kg_{r+s}
\end{equation}
such that $\kg_{r+1},\ldots,\kg_{r+s}$ are all 1-dimensional and
\begin{equation}\label{metric}
g= \alpha Q|_{\ag} + \alpha_1 Q|_{\kg_1} + \cdots + \alpha_{r+s} Q|_{\kg_{r+s}}
\end{equation}
with positive $\alpha_{r+1},\ldots,\alpha_{r+s}$.

\begin{remark}
If the homogeneous space $G/K$ is strongly isotropy irreducible, i.e., the identity component $K_0$ of $K$ acts irreducibly on~$\ag$, then 
$$s=\dim\zg(\kg)=1,$$
as follows from the classification results in~\cite[Chapter~10,~\S6]{Helgason} and~\cite{Wlf84}. In the more general case where $K$, but not $K_0$, acts irreducibly, there are many examples with $s>1$; see~\cite{WZ91}.
\end{remark}

As proven in~\cite{DZ79}, a metric of the form~\eqref{metric} is  naturally reductive with respect to~$G\times K$. Indeed, one can identify $G$ with the quotient $(G\times K)/\Delta(K)$, where
$$\Delta K=\{(k,k)\mid k\in K\}\subset G\times K.$$
Observe that $g$ is $(G\times K)$-invariant since the action of $\Ad(K)$ on $\ggo$ is clearly by isometries of $Q$ and hence of~$g$. One can find an $\Ad(G)$-invariant non-degenerate bilinear form on $\ggo\oplus \kg$ such that its induced metric on $G=(G\times K)/\Delta K$ coincides with~$g$. Consequently,
\begin{align*}
g([Z,X],Y) + g(X,[Z,Y]) =0
\end{align*}
for all $X, Y, Z $ in $\ggo\oplus \kg$ that are $Q$-orthogonal to the Lie algebra of~$\Delta(K)$. By definition, this means $g$ is naturally reductive with respect to $G\times K$. One of the main results of~\cite{DZ79} implies the converse, i.e.,  any left-invariant naturally reductive metric on a simple Lie group $G$ (regardless of the chosen description of $G$ as a homogeneous space) has the form~\eqref{metric} for some subgroup~$K$.

\subsection{The Ricci curvature}\label{secRicci}

We now describe the Ricci curvature of~$g$. Suppose 
$$n=\dim\ag, \qquad d_i=\dim\kg_i, \qquad i=1,\ldots,r+s.$$ 
Clearly, $d_{r+1}=\cdots=d_{r+s}=1$. 
If $B_i$ is the Killing form of~$\kg_i$, there exists a constant $\kappa_i$ such that
\begin{align*}
B_i = -\kappa_i Q|_{\kg_i}.
\end{align*}
Notice that $0\le \kappa_i\le 1$. Clearly, $\kappa_i=0$ if and only if $\kg_i$ lies in~$\zg(\kg)$. If $\kappa_i=1$, then $\kg_i$ is an ideal of $\ggo$, which is impossible since $G$ is simple. We can thus assume that $0<\kappa_i<1$ for $i=1,\ldots,r$ and $\kappa_{r+1}=\cdots=\kappa_{r+s}=0$. Finally, suppose 
\begin{align}
A_i(X,Y) = -\sum_{j=1}^{d_i}Q([X,v_j],[Y,v_j])\quad \text{ for all } X,Y\in\ag,
\end{align}
where $v_j$ is a $Q$-orthonormal basis of $\kg_i$. In~\cite{DZ79}, one finds the following formulas.

\begin{proposition}[D'Atri--Ziller]\label{Ric} Let $g$ be a metric on $G$ as in~(\ref{metric}). Its Ricci curvature is given by 
	\begin{align*}\label{Ricci1}
	&\Ricci_{g}(\ag,\kg_i)=\Ricci_{g}(\kg_i,\kg_j)=0, & 1 & \leq i,j \leq r+s,~i\ne j,\notag\\ 
	&\Ricci_{g}|_{\kg_i} = \unc\left(\kappa_i\bigg(1 - \frac{\alpha_i^2}{\alpha^2}\bigg) + \frac{\alpha_i^2}{\alpha^2}\right)Q|_{\kg_i},   & 1 & \leq i \leq r+s, \notag\\ 
	&\Ricci_{g}|_{\ag} = \unm \sum_{i=1}^{r+s} \left(\frac{\alpha_i}{\alpha} -1\right)A_i  + \unc  Q|_{\ag}. &
	\end{align*}
\end{proposition}

Let
$$\ag=\ag_1\oplus\cdots \oplus \ag_l$$
be a decomposition of $\ag$ into irreducible $\Ad(K)$-modules. Since $A_i$ is $\Ad(K)$-invariant, the restriction $A_i |_{\ag_j}$ equals $a_{ij} Q|_{\ag_j}$ for some constant~$a_{ij}$. This means $\Ricci_{g}|_{\ag_j}$ can be a different multiple of $Q|_{\ag_j}$ for each $j=1,\cdots,l$. Furthermore, $\Ricci_g(\ag_i,\ag_j)$ can be non-zero if $i\ne j$ and the $\Ad(K)$-modules $\ag_i$ and $\ag_j$ are equivalent. On the other hand, the restriction $g|_\ag$ is a multiple of~$Q$. This shows that the equation $\Ricci_g=cT$ is highly overdetermined unless $\ag$ is $\Ad(K)$-irreducible. One does not, in general, expect it to have solutions in this case. It is, therefore, natural for us to assume that $\ag$ is $\Ad(K)$-irreducible. Then 
$$A_i=- \frac{d_i(1-\kappa_i)}n Q|_{\ag}$$
(see~\cite[pages~34 and~46]{DZ79}), and the third equation in Proposition~\ref{Ric} simplifies to
\begin{equation}\label{Rica_irred}
\Ricci_{g}|_{\ag} = \bigg(-\unm \sum_{i=1}^{r+s} \left(\frac{\alpha_i}{\alpha}-1\right)\frac{d_i(1-\kappa_i)}n  + \unc \bigg) Q|_{\ag}.
\end{equation}

We will also need the following formula for the scalar curvature of~$g$. Due to Proposition~\ref{Ric} and an equality on page~34 of~\cite{DZ79}, we have:

\begin{corollary}\label{scalar}
If $g$ be a metric on $G$ as in~(\ref{metric}), then its scalar curvature is given by
	\begin{equation*}
	S_{g} = -\unc \sum_{i=1}^{r+s} \frac{\alpha_i}{\alpha^2} d_i (1-\kappa_i) + \unm \sum_{i=1}^{r+s} \frac{d_i (1-\kappa_i)}{\alpha} + \frac{n}{4 \alpha}+ \unc \sum_{i=1}^{r} \frac{\kappa_i d_i}{\alpha_i}.
	\end{equation*} 
\end{corollary}


\section{The prescribed Ricci curvature problem}\label{sec_results}

Let $\mathcal M_K$ be the set of left-invariant metrics on $G$ satisfying~\eqref{metric_gen} for some inner product $h$ on $\zg(\kg)$ and some positive constants $\alpha,\alpha_1,\ldots,\alpha_r$. As explained above, $g$ lies in $\mathcal M_K$ if and only if $g$ is naturally reductive with respect to~$G\times K$. Our objective in this section is to obtain conditions on $T\in\mathcal M_K$ that ensure the existence of $g\in\mathcal M_K$ such that
\begin{equation}\label{prescribed}
\Ricci_g = cT
\end{equation}
for some $c>0$. Most of our results assume that $\Ad(K)$ acts irreducibly on $\ag$. Otherwise, as explained above, one does not expect such a $g$ to exist. 

Denote by $\mathcal M_K^{\mathcal D}$ the set of metrics in $\mathcal M_K$ that are diagonal with respect to a decomposition $\mathcal D$ of the form~\eqref{centre_dec}. Every $g\in\mathcal M_K^{\mathcal D}$ satisfies~\eqref{metric} for some positive $\alpha,\alpha_1,\ldots,\alpha_{r+s}$. According to Proposition \ref{Ric}, the Ricci curvature of a metric in $\mathcal M_K^{\mathcal D}$ must be diagonal with respect to~$\mathcal D$ as well. Thus, if~\eqref{prescribed} holds for $g\in\mathcal M_K$ and $T\in\mathcal M_K$, then one can find $\mathcal D$ such that both $g$ and $T$ lie in~$\mathcal M_K^{\mathcal D}$.  

\subsection{The variational principle}\label{subsec_var}

In \cite{AP16}, the second-named author showed that solutions to the prescribed Ricci curvature problem on homogeneous spaces could be characterized as critical points of the scalar curvature functional subject to a constraint. We will now prove an analogous fact in our situation. Given $T\in\mathcal{M}^{\dca}_ K$, define
\begin{align}\label{MKTD_def}
\mathcal{M}^{\dca}_{K,T}=\big\{g \in \mathcal{M}^{\dca}_ K\,\big|\,\tr_{g}T=1\big\},
\end{align}
where $\tr_{g}T$ is the trace of $T$ with respect to $g$. In what follows, we view the scalar curvature as a functional $S:\mathcal M_K\to\mathbb R$.

\begin{proposition}\label{variational_lemma}
Assume $\Ad(K)|_{\ag}$ is irreducible. Then a metric $g \in\mathcal{M}^{\dca}_{K,T}$ satisfies (\ref{prescribed}) for some $c \in \RR$ if and only if it is a critical point of~$S |_{\mathcal{M}^{\dca}_{K,T}}$.
\end{proposition}

\begin{proof}
Let $\mathcal{M}$ and $\mathcal{T}$ be the set of $\hat G$-invariant metrics and the set of $\hat G$-invariant symmetric (0,2)-tensor fields on a general homogeneous space $\hat G/\hat K$. Clearly, the tangent space $T_u\mathcal{M}$ coincides with~$\mathcal{T}$ for any~$u\in\mathcal M$. The differential of the scalar curvature functional $\hat S:\mathcal M\to\mathbb R$ is 
\begin{align}\label{dif_scal}
d\hat S_u(v)=-\langle\Ricci_u,v\rangle,\qquad u\in\mathcal M,~v\in\mathcal T,
\end{align}
where the angular brackets denote the inner product on $\mathcal T$ induced by~$u$; see~\cite[page~277]{AP16}. Consequently, the gradient of this functional equals~$-\Ricci_u$ at~$u\in\mathcal M$.

Let ${\mathcal{T}^{\dca}_{K}}$ be the space of left-invariant (0,2)-tensor fields satisfying~\eqref{metric} for some constants $\alpha,\alpha_1,\ldots,\alpha_{r+s}$ (not necessarily positive). The tangent space $T_g {\mathcal{M}^{\dca}_{K}}$ coincides with $\mathcal{T}^{\dca}_{K}$ for any~$g\in\mathcal{M}^{\dca}_{K}$. Proposition~\ref{Ric} implies that $\Ricci_g$ lies in $\mathcal{T}^{\dca}_{K}$ for every~$g\in \mathcal{M}^{\dca}_{K}$.
Consequently, the gradient of $S|_{\mathcal{M}^{\dca}_{K}}$ is tangent to ${\mathcal{M}^{\dca}_{K}}$. The claim now follows from~\eqref{dif_scal} and the equality 
$$T_g {\mathcal{M}^{\dca}_{K,T}}=\big\{h\in {\mathcal{T}^{\dca}_{K}} \,\big|\, \langle T,h\rangle=0\big\},$$
where the angular brackets denote the inner product on ${\mathcal{T}^{\dca}_{K}}$ induced by~$g$.
\end{proof}

\begin{remark}
The variational characterization in Proposition~\ref{variational_lemma} does not hold if we allow $\Ad(K)|_{\ag}$ to be reducible since in this case $\Ricci_g$ is not necessarily tangent to~$\mathcal{M}^{\dca}_{K}$. The proof also shows that solutions to~(\ref{prescribed}) on ${\mathcal{M}_{K}}$ are again critical points of $S |_{\mathcal{M}_{K,T}}$, where $\mathcal{M}_{K,T}=\big\{g \in \mathcal{M}_K\,\big|\,\tr_{g}T=1\big\}$, if $K$ acts irreducibly. However, in Proposition~\ref{variational_lemma}, we diagonalize both $g$ and $T$ on~$\zg(\ggo)$.
\end{remark}

\subsection{The sufficient condition}\label{suf}

Fix $T\in\mathcal M_K$. The formula
\begin{equation}\label{T_gen}
T = T_{\ag} Q|_{\ag} + w + T_1 Q|_{\kg_1}  + \cdots + T_r Q|_{\kg_r}
\end{equation}
holds for some inner product $w$ on $\zg(\kg)$ and some positive constants $T_{\ag}, T_1 ,\ldots, T_r$. We can diagonalise $T$ with respect to a decomposition $\mathcal D$ of the form~\eqref{centre_dec}. Then $T$ lies in $\mca_K^{\dca}$ and
\begin{equation}\label{T}
T = T_{\ag} Q|_{\ag}   + T_1 Q|_{\kg_1}  + \cdots + T_{r+s} Q|_{\kg_{r+s}}
\end{equation}
for positive constants $T_{\ag}, T_1 ,\ldots, T_{r+s}$. Our next result shows that a simple inequality guarantees the existence of $g\in\mca_{K,T}^{\dca}$ satisfying~\eqref{prescribed}. The proof follows the strategy developed in~\cite{MGAP18,AP19}. We denote
\begin{align*}
d=n+d_1+\cdots+d_{r+s}=\dim\ggo.
\end{align*}

\begin{theorem}\label{sufcon}
Let $\Ad(K)|_{\ag}$ be irreducible. Consider a left-invariant naturally reductive metric $T \in \mathcal{M}_K^\dca$ satisfying~(\ref{T}). Choose an index $m$ such that 
\begin{align*}
\frac{\kappa_m}{T_m}=\max_{i=1,\ldots,r} \frac {\kappa_i}{T_i}.
\end{align*}
If 
\begin{align}\label{hyp_thm}
\frac{\kappa_m\tr_Q T}{T_m}< d + d_m - \kappa_md_m,
\end{align}
then the functional $S |_{\mathcal{M}^{\dca}_{K,T}}$ attains its global maximum at some $g_{\max} \in\mathcal{M}^{\dca}_{K,T}$. 
\end{theorem}

The proof of Theorem~\ref{sufcon} requires the following estimate for $S |_{\mathcal{M}^{\dca}_{K,T}}$.

\begin{lemma}\label{lemmacompact}
Given $\epsilon > 0$, there exists a compact set $\mathcal C_\epsilon\subset\mathcal{M}^{\dca}_{K,T}$ such that
\begin{equation}\label{estimate}
S_g < \frac{\kappa_m}{4T_m} + \epsilon
\end{equation}
for every $g \in \mathcal{M}^{\dca}_{K,T} \setminus\cca_\epsilon$.
\end{lemma}

\begin{proof}
Suppose $g \in \mathcal{M}^{\dca}_{K,T}$ satisfies~\eqref{metric}. The equality $\tr_{g}T=1$ implies
\begin{align}\label{tr_expl}
\frac{nT_\ag}{\alpha}+\sum_{i=1}^{r+s}\frac{d_iT_i}{\alpha_i}=1.
\end{align}
Consequently,
\begin{align}\label{simple_tr_est}
\alpha>nT_\ag,\qquad \alpha_i>d_iT_i,\qquad i=1,\ldots,r+s.
\end{align}
Assume 
\begin{align*}
\alpha>\Gamma_\ag(\epsilon)&=\max\bigg\{\frac{n}{2\epsilon},\frac{\sum_{i=1}^{r+s}d_i(1-\kappa_i)}{\epsilon}\bigg\}.
\end{align*}
Using~\eqref{tr_expl}, we find
\begin{align*}
S_g&\le\unm \sum_{i=1}^{r+s} \frac{d_i (1-\kappa_i)}{\alpha} + \frac{n}{4 \alpha}+ \unc \sum_{i=1}^{r} \frac{\kappa_i d_i}{\alpha_i} \\
&<\unm \sum_{i=1}^{r+s} \frac{d_i (1-\kappa_i)}{\Gamma_\ag(\epsilon)} + \frac{n}{4 \Gamma_\ag(\epsilon)}+ \frac{\kappa_m}{4T_m} \sum_{i=1}^{r} \frac{d_iT_i}{\alpha_i} \le \frac{\kappa_m}{4T_m}+\epsilon.
\end{align*}
Next, assume $\alpha\le\Gamma_\ag(\epsilon)$ and
\begin{align*}
\alpha_j>\Gamma_j(\epsilon)&=\frac{2\Gamma_\ag(\epsilon)^2}{T_\ag d_j(1-\kappa_j)}\bigg(\sum_{i=1}^{r+s} \frac{d_i (1-\kappa_i)}{n} + \unm\bigg)
\end{align*}
for some $j$ between 1 and $r+s$. By virtue of \eqref{tr_expl} and~\eqref{simple_tr_est},
\begin{align*}
S_g&<-\unc \sum_{i=1}^{r+s} \frac{\alpha_i}{\Gamma_\ag(\epsilon)^2} d_i (1-\kappa_i) + \unm \sum_{i=1}^{r+s} \frac{d_i (1-\kappa_i)}{nT_\ag} + \frac1{4 T_\ag}+ \unc \sum_{i=1}^{r} \frac{\kappa_i d_i}{\alpha_i}
\\
&<-\unc \frac{\Gamma_j(\epsilon)}{\Gamma_\ag(\epsilon)^2} d_j (1-\kappa_j) + \frac1{2T_\ag}\bigg(\sum_{i=1}^{r+s} \frac{d_i (1-\kappa_i)}n + \unm\bigg)+ \frac{\kappa_m}{4T_m} \sum_{i=1}^{r} \frac{d_iT_i}{\alpha_i}\le \frac{\kappa_m}{4T_m}.
\end{align*}
Thus, we proved~\eqref{estimate} for metrics in $\mathcal{M}^{\dca}_{K,T}$ lying outside the set
\[
\cca_\epsilon = \big\{g \in \mathcal{M}^{\dca}_{K,T}  \,\big|\, \eqref{metric}~\mbox{holds with}~\alpha\le\Gamma_\ag(\epsilon)~\mbox{and}~\alpha_i\le\Gamma_i(\epsilon)~\mbox{for}~i=1,\ldots,r+s\big\}.
\]
It is easy to check that $\cca_\epsilon$ is compact; cf.~\cite[Lemma~2.24]{MGAP18}.
\end{proof}

\begin{proof}[Proof of Theorem \ref{sufcon}]
Denote $U= \tr_QT -d_mT_m$. For $t>U$, consider the metric $g_t\in\mca_{K}^\dca$ satisfying
\begin{align*}
g_t=  t Q|_{\ag} &+  t Q|_{\kg_1} + \cdots +  t Q|_ {\kg_{m-1}} 
\\ &+ \phi(t)Q|_{\kg_m} +  tQ|_ {\kg_{m+1}} + \cdots +   tQ|_ {\kg_{r+s}}, 
\qquad 
\phi(t) = \frac {d_mT_mt}{t- U}.
\end{align*}
Straightforward verification shows that 
$g_t$ lies in~$\mca_{K,T}^\dca$. By Corollary~\ref{scalar},
\begin{align*}
S_{g_t} =  &\frac1{4t}d_m (1-\kappa_m) -  \frac{\phi(t)}{4t^2} d_m (1-\kappa_m) \\
&+ \frac{1}{4t} \sum_{i=1}^{r+s} d_i (1-\kappa_i) + \frac{n}{4t} + \frac{1}{4t} \sum_{i=1}^{r} \kappa_i d_i - \frac{1}{4t}\kappa_md_m + \frac{\kappa_md_m}{4\phi(t)}.
\end{align*}
Furthermore, in light of~\eqref{hyp_thm},
\begin{align}\label{t2St}
4\lim_{t\to\infty}t^2\frac{d}{dt} S_{g_t} & = -d_m(1-\kappa_m) - \sum_{i=1}^{r+s} d_i (1-\kappa_i) - n - \sum_{i=1}^{r} \kappa_i d_i +\kappa_md_m+ \frac{\kappa_mU}{ T_m}  \notag 
\\
& = -(n+d_1+\cdots+d_{r+s}) -d_m +2\kappa_md_m+ \frac{\kappa_m( \tr_QT-d_mT_m)}{ T_m}  \notag
\\
& = -d - d_m + d_m \kappa_m + \frac{\kappa_m\tr_QT}{ T_m}<0.
\end{align}
We conclude that $\frac{d}{dt} S_{g_t}<0$ for sufficiently large $t$, which implies the existence of $t_0\in(U,\infty)$ such that
\[
S_{g_{t_0}} >\lim_{t \to \infty} S_{g_t} = \frac{\kappa_m}{4T_m}.
\]
Using Lemma \ref{lemmacompact} with 
\[
\epsilon =  \unm\Big(S_{g_{t_0}} - \frac{\kappa_m}{4T_m}\Big)>0
\] yields
\begin{equation}\label{maxcompact}
S_{g'} < \frac{\kappa_m}{4T_m}+\epsilon=\unm S_{g_{t_0}}+\frac{\kappa_m}{8T_m}< S_{g_{t_0}},\qquad g' \in \mathcal{M}^{\dca}_{K,T}\setminus\cca_\epsilon.
\end{equation} 
Since $\cca_\epsilon$ is compact, the functional $S|_{\mathcal{M}^{\dca}_{K,T}}$ attains its global maximum on $\cca_\epsilon$ at some ${g_{\max} \in \cca_\epsilon}$. Obviously, $g_{t_0}$ lies in $\cca_\epsilon$. Thus, in light of~\eqref{maxcompact},
\[
S_{g'} \leq S_{g_{\max}}
\] 
for all $g' \in \mathcal{M}^{\dca}_{K,T}$.
\end{proof}

Proposition~\ref{variational_lemma} implies the existence of a constant $c\in\mathbb R$ such that the Ricci curvature of $g_{\max}$ equals~$cT$. This constant is necessarily positive by ~\cite[Theorem~1.84]{Bss} and~\cite[Theorem~7.61]{Bss} since $G$ is simple.

\begin{remark}\label{degenerate}
Throughout Section~\ref{sec_results}, we assumed that $T$ was positive-definite. If it is only positive-semidefinite, one easily sees that~\eqref{prescribed} cannot be satisfied for any $g\in\mca_K$ unless $T_i\ne0$ for all $i=1,\ldots,r+s$. On the other hand, Theorem~\ref{sufcon} holds if $T_\ag=0$ and $T_1,\ldots,T_{r+s}$ are positive.
\end{remark}

\subsection{The necessary condition}\label{nec}

Our next result shows that solutions to \eqref{prescribed} do not always exist.

\begin{theorem}\label{neccon}
Let $\Ad(K)|_{\ag}$ be irreducible. Consider a left-invariant naturally reductive metric $T\in\mca_K$ satisfying~(\ref{T_gen}). If there exists $g\in\mathcal{M}_K$ such that~(\ref{prescribed}) holds, then
\begin{align}\label{ineq_nec_cond}
n T_\ag \max_{i=1,\ldots,r}\frac{\kappa_i}{T_i} < 2 \sum_{j=1}^{r} d_j (1-\kappa_j) + 2s + n.
\end{align}
\end{theorem}

\begin{proof}
As we saw in the beginning of Section~\ref{sec_results}, we can assume that both $g$ and $T$ lie in~$\mca_K^{\dca}$ for some fixed decomposition $\dca$ of the form~\eqref{centre_dec}. Let formulas~\eqref{metric} and~\eqref{T} hold. Given an integer $j$ between 1 and $r$, denote
\begin{align*}
V_j=-\ln\Big(1+\frac{n T_\ag\alpha_j}{d_j T_j\alpha}\Big).
\end{align*}
For $t>V_j$, consider the metric $g_t^j\in\mca_K^{\dca}$ satisfying
\begin{align*}
g_t^j&= f_j(t) Q|_{\ag}+\sum_{i=1}^{r+s}e^{\delta_i^jt}\alpha_i Q|_{\kg_i},\qquad 
f_j(t)=\frac{n T_\ag\alpha\alpha_j}{n T_\ag\alpha_j+d_j T_j\alpha-e^{-t}d_j T_j\alpha},
\end{align*}
where $\delta_i^j$ is the Kronecker symbol. Straightforward computation shows that $g_t^j$ lies in~$\mca_{K,T}^{\dca}$ and $g_0^j$ coincides with~$g$. We thus obtain a curve $\big(g_t^j\big)_{t>V_j}\subset\mca_{K,T}^{\dca}$ passing through $g$ at~$t=0$.

By Proposition~\ref{variational_lemma}, $g$ is a critical point of $S|_{\mca_{K,T}^{\dca}}$. Consequently,
\begin{align}\label{ddtS=0}
\frac d{dt}S_{g_t^j}\Big|_{t=0}=0.
\end{align}
Corollary~\ref{scalar} yields
\begin{align*}
S_{g_t^j}&=-\frac1{4f_j(t)^2}\sum_{i=1}^{r+s}e^{\delta_i^jt}\alpha_i d_i (1-\kappa_i)
+ \frac1{2f_j(t)}\sum_{i=1}^{r+s} d_i (1-\kappa_i)
+ \frac n{4f_j(t)}+ \unc \sum_{i=1}^{r} \frac{\kappa_i d_i}{e^{\delta_i^jt}\alpha_i},
\end{align*}
which means
\begin{align*}
\frac d{dt}S_{g_t^j}\Big|_{t=0}&=-\frac{d_jT_j}{2nT_\ag \alpha \alpha_j}\sum_{i=1}^{r+s}\alpha_i d_i (1-\kappa_i) \\
&\hphantom{=}~- \frac{\alpha_j}{4\alpha^2} d_j (1-\kappa_j) +\frac{d_jT_j}{2nT_\ag\alpha_j}\sum_{i=1}^{r+s} d_i (1-\kappa_i) 
+ \frac{d_jT_j}{4T_\ag\alpha_j} - \frac{\kappa_j d_j}{4\alpha_j}.
\end{align*}
The first two terms on the right-hand side are negative. Therefore, by virtue of~\eqref{ddtS=0},
\begin{align*}
\frac{d_jT_j}{2nT_\ag\alpha_j}\sum_{i=1}^{r+s} d_i (1-\kappa_i) + \frac{d_jT_j}{4T_\ag\alpha_j} - \frac{\kappa_j d_j}{4\alpha_j}>0.
\end{align*}
Multiplying by $\frac{4nT_\ag\alpha_j}{d_jT_j}$ and rearranging, we obtain
\begin{align*}
nT_\ag\frac{\kappa_j}{T_j}<2\sum_{i=1}^{r} d_i (1-\kappa_i) + 2s + n.
\end{align*}
Since this holds for any $j=1,\ldots,r$, the assertion of the theorem follows.
\end{proof}

\begin{remark}
As we saw, if $\Ad(K)|_{\ag}$ is reducible, the system~\eqref{prescribed} is overdetermined in general. In this case, it is, perhaps, more natural to consider the problem of finding $g\in\mathcal M_K$ that satisfies
\begin{align}\label{alternative_prescribed}
\Ricci_{g}|_{\kg}=cT|_{\kg}, \qquad
\tr_Q\Ricci_{g}|_{\ag}=c\tr_QT|_{\ag}.
\end{align}
Arguing as in the proof of Proposition~\ref{variational_lemma}, one can show that~\eqref{alternative_prescribed} holds for $g\in\mathcal M_{K,T}^\dca$ if and only if $g$ is a critical point of~$S|_{\mca_{K,T}^\dca}$. With the methods developed above, one easily obtains analogues of Theorems~\ref{sufcon} and~\ref{neccon} for problem~\eqref{alternative_prescribed}.
\end{remark}

\begin{remark}\label{semisimple}
Throughout Section~\ref{sec_results}, we have assumed the group $G$ was simple. This hypothesis can be weakened. More precisely, suppose that $G$ is compact semisimple and that $\mathfrak g$ and $\mathfrak k$ do not share any non-trivial common ideals. Let $Q$ be a bi-invariant metric on~$G$. One can then show that Theorems~\ref{sufcon} and~\ref{neccon} hold with the same proof as above. Every $T$ of the form~\eqref{T_gen} (and hence~\eqref{T}) is naturally reductive with respect to~$G\times K$; however, the converse may be false if $G$ is not simple (see~\cite[pages~9 and~20]{DZ79}). 
\end{remark}

\section{The case where $K$ is simple}\label{Ksimple}

In this section, we assume that $K$ is simple and $\Ad(K)|_{\ag}$ is irreducible. Then Theorems~\ref{sufcon} and~\ref{neccon} yield conditions for the solvability of~\eqref{prescribed} that are at the same time necessary and sufficient.

Suppose $T\in\mca_K$. The numbers $r$ and $s$ in decompositions~\eqref{deck} and~\eqref{centre_dec} equal 1 and~0, respectively. Thus,
\begin{align}\label{T2}
T = T_{\ag} Q|_{\ag}   + T_1 Q|_{\kg_1}
\end{align}
for two constants $T_\ag,T_1>0$.

\begin{proposition}\label{prop_simple_gr}
Assume $K$ is simple and $\Ad(K)|_{\ag}$ is irreducible. Given $T\in \mathcal M_K$, a metric $g \in \mathcal{M}_{K}$ satisfying~\eqref{prescribed} for some~$c>0$ exists if and only if 
\begin{align}\label{cond_simple_gr}
(2 d_1 (1-\kappa_1)+n)T_1>n T_{\ag}\kappa_1.
\end{align}
When it exists, this metric is unique up to scaling.
\end{proposition}

\begin{proof}
Since $r=1$ and $s=0$, we can transform~\eqref{hyp_thm} and~\eqref{ineq_nec_cond} into~\eqref{cond_simple_gr} by elementary computations. In light of this fact, the equivalence of the statements in the proposition follows from Theorems~\ref{sufcon} and~\ref{neccon}. It remains to prove uniqueness up to scaling. Take $g\in\mca_K$ and $g'\in\mca_K$ with
\begin{align*}
\Ricci_{g}=cT,\qquad \Ricci_{g'}=c'T,\qquad c,c'>0.
\end{align*}
Without loss of generality, assume $c\ge c'$. There exist $\alpha,\alpha_1,\alpha',\alpha_1'$ such that
\[
g = \alpha Q|_{\ag}   + \alpha_1 Q|_{\kg_1},\qquad g' = \alpha' Q|_{\ag}   + \alpha_1' Q|_{\kg_1}.
\] 
The second of the three formulas in Proposition~\ref{Ric} implies
\begin{align*}
\tfrac14(1-\kappa_1)\Big(\frac{\alpha_1}{\alpha} - \frac{\alpha_1'}{\alpha'}\Big)\Big(\frac{\alpha_1}{\alpha} + \frac{\alpha_1'}{\alpha'}\Big)=(c-c')T_1,
\end{align*}
which means $\tfrac{\alpha_1}{\alpha} \ge \tfrac{\alpha_1'}{\alpha'}$. At the same time, by~\eqref{Rica_irred},
\begin{equation*}
 -\unm \frac{d_1(1-\kappa_1)}{n}\bigg(\frac{\alpha_1}{\alpha}-\frac{\alpha_1'}{\alpha'}\bigg) = (c - c')T_\ag,
\end{equation*}
whence $\tfrac{\alpha_1}{\alpha} \le \tfrac{\alpha_1'}{\alpha'}$. We conclude that $g$ and $g'$ coincide up to scaling.
\end{proof}

\begin{remark}
It is easy to see that Proposition~\ref{prop_simple_gr} continues to hold if we assume $T$ is a degenerate left-invariant tensor field on $G$ satisfying~\eqref{T2} with $T_\ag,T_1\ge0$.
\end{remark}

\section{Examples}\label{examples}

We now discuss a series of examples that illustrate our results and enable us to draw the graph of the scalar curvature functional. Assume that $G=\SO(6)$ and $K=\SO(3)\times\SO(3)$, embedded diagonally. Then $G/K$ is an isotropy irreducible symmetric space. The center of $\kg$ equals $\{0\}$, which means the only decomposition $\dca$ of the form~\eqref{centre_dec} is the trivial one, and hence $\mathcal M_K$ and $\mathcal M_K^\dca$ coincide.
The simple ideals of $\kg$ are $\kg_1=\sog(3)\oplus\{0\}$ and $\kg_2=\{0\}\oplus\sog(3)$. Thus,
$$r=2,\qquad s=0,\qquad d_1=d_2=3,\qquad n=9.$$
 The constants $\kappa_1$ and $\kappa_2$ both equal~$\unc$; see~\cite[Page 51, Example 1]{DZ79}.
 
Given  $T\in\mca_K$, there exist positive constants $T_\ag,T_1,T_2$ such that
\begin{align*}
T=T_\ag Q|_{\ag}+T_1Q|_{\kg_1}+T_2Q|_{\kg_2}.
\end{align*}
The solvability properties of~\eqref{prescribed} do not change if one rescales~$T$. Therefore, we may assume without loss of generality that~$T_\ag=1$. The sufficient condition of Theorem~\ref{sufcon} becomes
\begin{align}\label{suf_ex}
23\min\{T_1,T_2\}>3+T_1+T_2,
\end{align}
and the necessary conditions of Theorem~\ref{neccon} becomes
\begin{align}\label{nec_ex}
8\min\{T_1,T_2\}>1.
\end{align}
Figure~\ref{T_chart_T1T2} shows the sets of points $(T_1,T_2)$ satisfying these inequalities. The cylindrical algebraic decomposition (CAD) algorithm---see \cite{CD98}---predicts the existence of a solution to~\eqref{prescribed} if and only if
$$\frac{12+T_1+\sqrt{T_1^2+24T_1-3}}{98} < T_2 < \frac{196 T_1^2-48T_1+3}{4 T_1}.$$
The ``middle" region in Figure~\ref{T_chart_T1T2} is the set of $(T_1,T_2)$ for which this condition holds but~\eqref{suf_ex} does not. The red dots mark the choices of $(T_1,T_2)$ considered in Examples~\ref{example_suf}--\ref{example_funny_max} below.

\begin{figure}[h]\label{T_chart_T1T2} 
\includegraphics[width=0.63\textwidth]{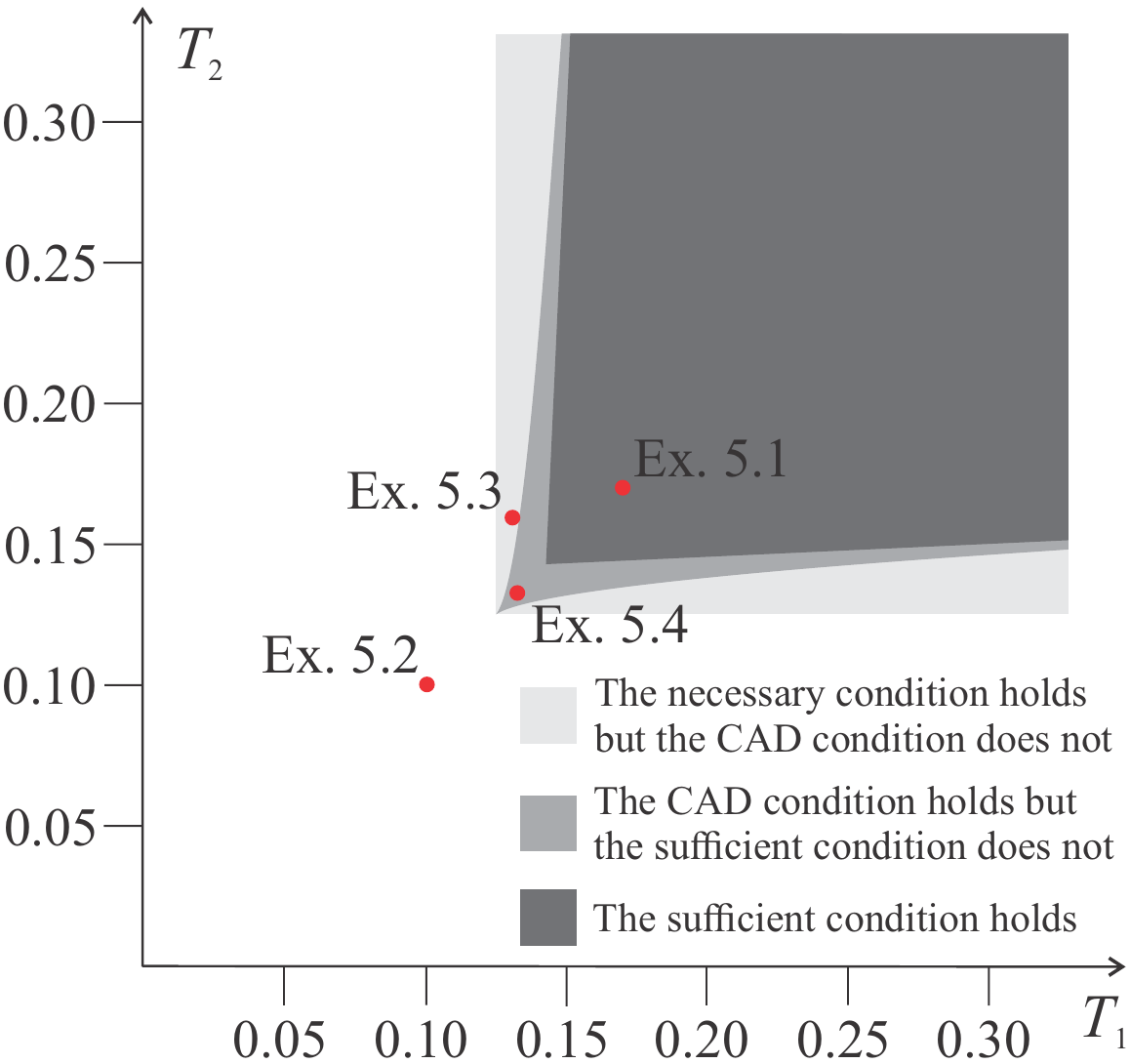}
\caption{Necessary, sufficient and CAD conditions}\label{T_chart_T1T2}
\end{figure}

By Proposition~\ref{variational_lemma}, 
 $g\in\mathcal{M}_{K,T}$ satisfies~\eqref{prescribed} for some $c>0$ if and only if it is a critical point of the scalar curvature functional~$S|_{\mca_{K,T}}$. Suppose $g\in\mathcal{M}_{K,T}$ is given by
\begin{align}\label{dec_g_3}
g&=\alpha Q|_{\ag}+\alpha_1Q|_{\kg_1}+\alpha_2Q|_{\kg_2},\qquad \alpha,\alpha_1,\alpha_2>0.
\end{align}
The condition $\tr_gT=1$ becomes
\begin{align*}
\frac{9}{\alpha}+\frac{3T_1}{\alpha_1}+\frac{3T_2}{\alpha_2}=1,\ \text{ and thus }\
\frac1\alpha=\frac{\alpha_1\alpha_2-3T_1\alpha_2-3T_2\alpha_1}{9\alpha_1\alpha_2}.
\end{align*}
Hence, Corollary~\ref{scalar} implies
\begin{align*}
S_{g}
=-\frac{(\alpha_1\alpha_2-3T_1\alpha_2-3T_2\alpha_1)^2(\alpha_1+\alpha_2)}{144\alpha_1^2\alpha_2^2} + \frac{\alpha_1\alpha_2-3T_1\alpha_2-3T_2\alpha_1}{2\alpha_1\alpha_2} + \frac{3}{16\alpha_1}+\frac{3}{16\alpha_2}.
\end{align*}
This enables us to view $S|_{\mathcal{M}_{K,T}}$ as a function of the two variables $\alpha_1$ and $\alpha_2$ defined on the set
\begin{align*}
\Big\{(x,y)\in(0,\infty)^2 \,\Big|\, \frac{3T_1}x+\frac{3T_2}y<1\Big\}.
\end{align*}
Given $(T_1,T_2)$, we  use Maple to draw the graph of~$S|_{\mathcal{M}_{K,T}}$.

\begin{example}\label{example_suf} 
Suppose $(T_1,T_2)=\left(\frac{1}{6},\frac{1}{6}\right)$. Clearly, condition~\eqref{suf_ex} is satisfied. Theorem~\ref{sufcon} implies that $S|_{\mathcal{M}_{K,T}}$ assumes its global maximum at some metric~$g_{\max}$ with Ricci curvature equal to~$cT$. Figure~\ref{Figure_globalmax_sufnec} shows the graph of $S|_{\mca_{K,T}}$ as a function of $\alpha_1$ and~$\alpha_2$. The red dot marks the global maximum. 
\begin{figure}[h]\label{fgm1} 
\includegraphics[width=0.8\textwidth]{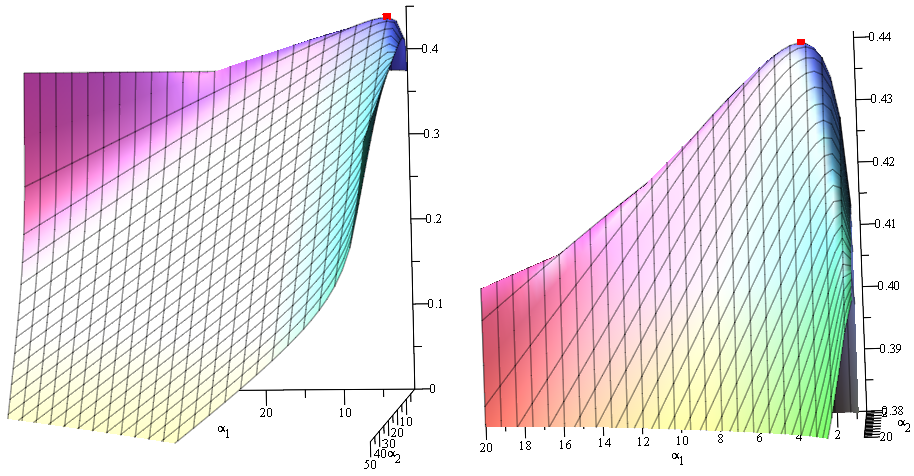}
\caption{Global maximum with~\eqref{suf_ex} satisfied.}\label{Figure_globalmax_sufnec}
\end{figure}
\end{example}

\begin{example}\label{example_nec}
Suppose  $(T_1,T_2)=\left(\frac{1}{10},\frac{1}{10}\right)$. In this case, condition~\eqref{nec_ex} fails to hold, and hence Theorem~\ref{neccon} implies that  $S|_{\mca_{K,T}}$ has no critical points. Figure~\ref{Figure_ncp_notnec} shows the graph of $S|_{\mca_{K,T}}$ as a function of $\alpha_1$ and~$\alpha_2$.
\begin{figure}[h]\label{fgm2} 
\includegraphics[width=0.4\textwidth]{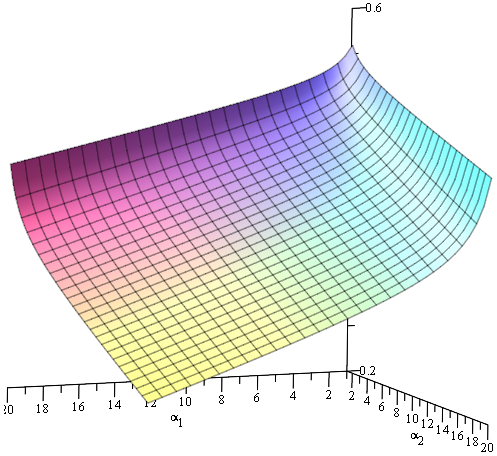}
\caption{No critical point with~\eqref{nec_ex} failing to hold.}\label{Figure_ncp_notnec}
\end{figure}
\end{example}

The necessary and sufficient conditions in Section~\ref{sec_results} do not cover the situation where~\eqref{nec_ex} is satisfied but~\eqref{suf_ex} is not. As the following two examples demonstrate, in this situation, the functional $S|_{\mca_{K,T}}$ may exhibit a variety of behaviors. 
\begin{example}\label{example_ncp}
Suppose $(T_1,T_2)=\left(\frac{13}{100},\frac{16}{100}\right)$. In this case,~\eqref{nec_ex} holds, but~\eqref{suf_ex} does not. Straightforward analysis proves that $S|_{\mca_{K,T}}$ has no critical points. Figure~\ref{Figure_ncp_notsuf} shows the graph of~$S|_{\mca_{K,T}}$.
\begin{figure}[h]\label{fgm3} 
\includegraphics[width=0.4\textwidth]{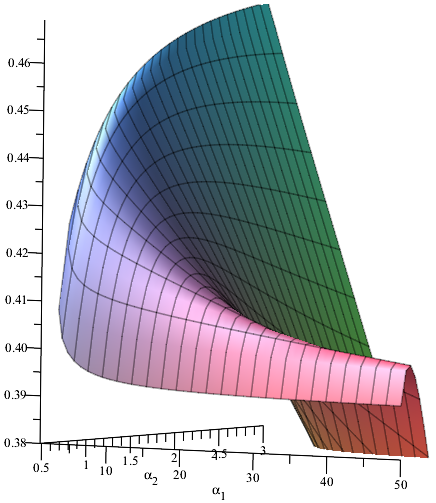}
\caption{No critical point with~\eqref{nec_ex} satisfied.}\label{Figure_ncp_notsuf}
\end{figure}
\end{example}

\begin{example}\label{example_funny_max}
Suppose $(T_1,T_2)=\left(\frac{2}{15},\frac{2}{15}\right)$. Again,~\eqref{nec_ex} holds, but~\eqref{suf_ex} does not. Figure~\ref{Figure_globmax_notsuf} shows the graph of~$S|_{\mca_{K,T}}$. Applying Lemma~\ref{lemmacompact} with $\epsilon=\tfrac1{192}$, we find
\begin{align*}
S_g<\max\Big\{\frac{\kappa_1}{4T_1},\frac{\kappa_2}{4T_2}\Big\}+\epsilon=\tfrac{91}{192}
\end{align*}
when the metric ${g\in\mca_{K,T}}$ lies outside some compact subset of~$\mca_{K,T}$. At the same time, 
\begin{align*}
S_g=\tfrac{427}{900}>\tfrac{91}{192}
\end{align*}
if $g$ is given by~\eqref{dec_g_3} with $(\alpha,\alpha_1,\alpha_2)=(45,1,1)$. Consequently, $S|_{\mca_{K,T}}$ attains its global maximum at some~${g_{\max}\in\mca_{K,T}}$.  The red dot marks this global maximum on the graph.
\begin{figure}[h]\label{fgm4} 
\includegraphics[width=0.8\textwidth]{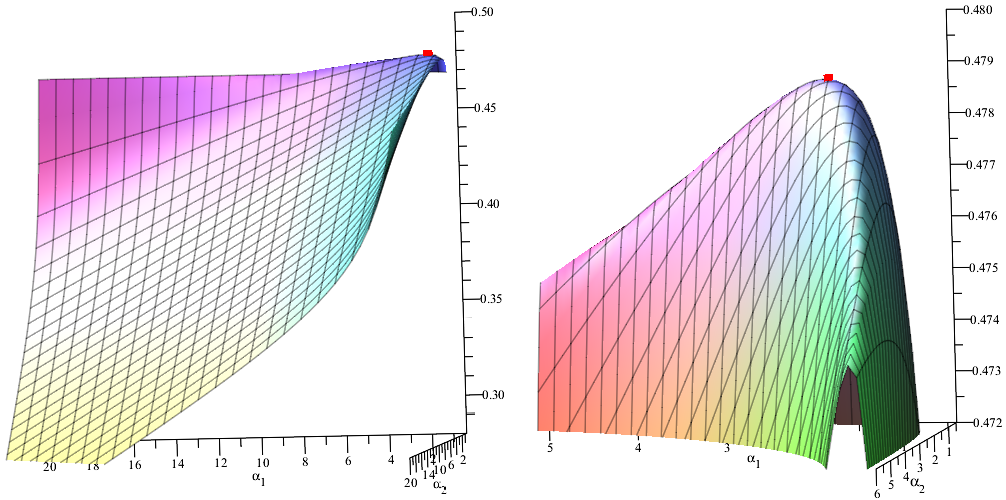}
\caption{Global maximum with~\eqref{suf_ex} failing to hold.}\label{Figure_globmax_notsuf}
\end{figure}
\end{example}

Example~\ref{example_funny_max} is particularly interesting since the behavior that $S|_{\mca_{K,T}}$ exhibits has not been observed in any previous work on the prescribed Ricci curvature problem on homogeneous spaces. In~\cite{MGAP18,AP19}, the existence of a global maximum was proven by showing that the scalar curvature decreased monotonically as the components of the metric approached infinity. Although barely visible, in Example~\ref{example_suf}, the scalar curvature decreases when $\alpha_1$ goes to $\infty$ along the ridge of the graph, whereas in Example~\ref{example_funny_max}, it increases. Nevertheless, $S|_{\mca_{K,T}}$ attains its global maximum in both examples.


\end{document}